\documentclass[12pt]{amsart}
\font\bbbld=msbm10 scaled\magstephalf

\newcommand{\mcH}{\mathcal{H}}
\newcommand{\bfH}{\hbox{\bbbld H}}
\newcommand{\bfN}{\hbox{\bbbld N}}
\newcommand{\bfR}{\hbox{\bbbld R}}
\newcommand{\bfS}{\hbox{\bbbld S}}
\newcommand{\bfF}{\hbox{\bbbld F}}
\newcommand{\bfC}{\hbox{\bbbld C}}
\newcommand{\bfO}{\hbox{\bbbld O}}
\newcommand{\bfFP}{\hbox{\bbbld FP}}
\newcommand{\goto}{\rightarrow}

\newcommand{\be}{\begin{equation}}
\newcommand{\ee}{\end{equation}}
\newcommand{\e}{\varepsilon}

\newtheorem{theorem}{Theorem}[section]
\newtheorem{lemma}[theorem]{Lemma}
\newtheorem{proposition}[theorem]{Proposition}
\newtheorem{corollary}[theorem]{Corollary}

\theoremstyle{definition}

\theoremstyle{remark}

\numberwithin{equation}{section}

\begin{document}
\setlength{\baselineskip}{1.2\baselineskip}

\title[Self-shrinkers to the mean curvature flow]
{Self-shrinkers to the mean curvature flow asymptotic to isoparametric cones}

\author{Po-Yao Chang}
\address{Department of Mathematics, Johns Hopkins University,
 Baltimore, MD, 21218}
\email{pchang@math.jhu.edu}
\author{Joel Spruck}
\address{Department of Mathematics, Johns Hopkins University,
 Baltimore, MD, 21218}
\email{js@math.jhu.edu}
\thanks{Research  supported in part by the NSF}


\begin{abstract} In this paper we  construct an end  of a self-similar shrinking solution of the mean curvature flow  asymptotic to an isoparametric cone C and lying outside of C. We call a cone C in $\bfR^{n+1}$  an  {\em isoparametric cone} if C is the cone over a compact embedded isoparametric hypersurface $\Gamma \subset \bfS^n$. The theory of isoparametic hypersurfaces is extremely rich  and there are  infinitely many distinct classes of examples, each with infinitely many members.
 \end{abstract}

\maketitle
\section{Introduction}
\label{sec0}
\setcounter{equation}{0}

A hypersurface $\Sigma$ in $\bfR^{n+1}$ is said to be a self-shrinker (centered at (0,0) of space-time) for the mean curvature flow  if
$\Sigma_t=\sqrt{-t}\,\Sigma$  flows by homothety starting at $t=-1$ until it disappears at time $t=0$. A simple computation shows that $\Sigma_t$ is a self-shrinker if and only if
 $\Sigma$  satisfies the equation
\be \label{eq0.10} H=-\frac12 X\cdot \nu~.
\ee
where $H$ is the mean curvature,  $ X$ is the position vector and $\nu$ is the unit normal of $\Sigma$.
The study of self-similar shrinking solutions  to the mean curvature flow  is now well understood to be an important and essential feature of the classification of possible singularities that may develop. In fact the monotonicity formula of Huisken \cite{Hu} and a rescaling argument of Ilmanen and White imply that suitable blowups of singularities of the mean curvature flow are self-shrinkers \cite{CM}. By a theorem of Wang \cite{W1}, if $C$ is a smooth regular cone with vertex at the origin, there is at most one self-shrinker with an end asymptotic to $C$. \\

There are relatively few constructions in the literature of  self-shrinkers asymptotic to a cone C; see for example \cite{KM}, \cite{KKM}, \cite{N1, N2, N3}. 
In this paper we will construct, for infinitely many families of special mean convex cones $C$ with interesting topology,  a corresponding end of a self-shrinker which is asymptotic to $C$ and lies outside of C.
A closed connected compact hypersurface $\Gamma \subset \bfS^n \subset \bfR^{n+1}$ is called an isoparametric hypersurface if its principal curvatures are constant. Equivalently, $\Gamma$ is part of a family of
parallel hypersurfaces in $\bfS^n$ which have constant mean curvature.
 By a theorem of Cecil and Ryan \cite{CR}, $\Gamma$ is taut and so  is automatically embedded.
We will say that  a cone C in $\bfR^{n+1}$ is an  {\em isoparametric cone} if C is the cone over a compact embedded isoparametric hypersurface $\Gamma$. The theory of isoparametric hypersurfaces in $\bfS^n$ is extremely rich and beautiful. Cartan classified all isoparametric hypersurfaces in $\bfS^n$ with $g\leq 3$ distinct principal curvatures. For $g=1$ they are the totally umbilic hyperspheres while for $g=2$ they are a standard product of spheres $\bfS^p(a)\times \bfS^q(b)$ where $a^2+b^2=1$ and $n=p+q+1$. For $g=3$ Cartan showed that all the principal curvatures have the same multiplicity $m=1,2,4$ or $8$ and $\Gamma$ must be a tube of constant radius over a standard embedding of a projective plane $\bfFP^2$ into $\bfS^{3m+1}$ where $\bfF$ is the division algebra $\bfR,\,\bfC,\, \bfH \,\,\text{(quaternions)},\,\,\bfO\, \text{(Cayley numbers)}$. In the process of proving this result, he showed that any such $\Gamma$ with $g$ distinct
 principal curvatures of the same multiplicity can be defined by the restriction to $\bfS^n$ of a homogeneous  harmonic  polynomial $F$ of degree g on $\bfR^{n+1} $ satisfying
 $|\nabla F|^2=g^2 |x|^{2g-2}$. 
 M\"{u}nzner \cite{M1}, \cite{M2} found a remarkable structural generalization of this last result of Cartan.
Let $\Gamma \subset \bfS^n$ be an isoparametric hypersurface with $g$ distinct principal curvatures. Then there is a homogeneous polynomial $F$ of degree $g$  defined in all of $\bfR^{n+1}$ (the Cartan-M\"{u}nzner polynomial) satisfying
%\be \label{eq0.10} 
  $ |\nabla F|^2=g^2 r^{2g-2},\, \Delta F=\frac{m_{-}-m_{+}}2 g^2 r^{g-2},\, $ where $r=|x|$.
 The restriction $f$ of $F$ to $\bfS^n$ has range $[-1,1]$ and  satisfies
\[|\nabla_{S^n}f|^2=g^2(1-f^2),\, \Delta_{S^n}f+g(n+g-1)f=\frac{m_{-}-m_{+}}2 g^2~,\]
Each member $\Gamma_t$ of  the isoparametric family determined by $\Gamma$ has the same focal sets  $M_{\pm}:=f^{-1}(\pm1)$  which are smooth minimal submanifolds of codimension $m_{-}+1,\, m_{+}+1$ respectively. In proving this result, M\"{u}nzner shows that if the principal curvatures of $\Gamma$ are written as $\cot{\theta_k},\, 0<\theta_1<\ldots <\theta_g<\pi$, then $\theta_k=\theta_1+\frac{k-1}g \pi$ with multiplicities
$m_k=m_{k+2}\,\,\text{subscripts mod g}$. Thus for $g$ odd, all multiplicities are the same while for $g$ even, there are at most two distinct multiplicities $m_{-},\,m_{+}$. Moreover each $\Gamma_t$ separates $\bfS^n$ into two connected components $D_{\pm}$ such that $D_{\pm}$ is a disk bundle with fibers of dimension $m_{\pm}+1$ over $M_{\pm}$. From this he is able to deduce using algebraic topology that $g\in \{1,2,3,4,6\}$. Earlier, Takagi and Takahashi \cite{TT} had classified all the homogeneous examples based on the work of Hsiang and Lawson \cite{HL} and had found that $g\in \{1,2,3,4,6\}$.\\

For $g=4$ using representations of Clifford algebras, Ozeki and Takeuchi \cite{OT1, OT2}  found two classes of examples, each with infinitely many members of inhomogeneous solutions. Later these methods were greatly generalized by
Ferus, Karcher and M\"{u}nzner \cite{FKM}, who showed there are infinitely many distinct classes of solutions (in odd dimensional spheres $\bfS^{2l-1}\subset \bfR^{2l}$) each with infinitely many members.  Their examples contain almost all known homogeneous and inhomogeneous examples. For $g=6$ Abresch \cite{A} showed that $m_{-}=m_{+}=1\,\,\text{or 2}$ so examples occur only in dimensions $n=7$ or $13$.\\

From the isoparametric function $f$ we easily compute the mean curvature $H(f)$ of the level set $f=t$:
\be \label{eq0.30}
H(f)=\frac1{|\nabla_{S^n}f|}\Delta_{S^n}f=\frac1{\sqrt{1-f^2}}\{g(\frac{m_{-}-m_{+}}2)-(n+g-1)f\}~,
\ee
and so the level set $f=\frac{g(m_{-}-m_{+})}{2(n+g-1)}$ is the unique minimal isoparametric hypersurface of the family. Our main theorem may be stated as follows.
\begin{theorem}\label{th0.1} Let C be an isoparametric cone over $\Gamma \in \bfS^n$. Then there is a radial graph $\Sigma=\{e^{\varphi(d(z))}: z\in A\}$, where $d(z)$ is the distance function to $\Gamma,\, A=\{z: 0<d(z)\leq d_0+\e\}$,  which is an end of a  self-shrinker to the mean curvature  flow (that is satisfies \eqref{eq0.10}) and is asymptotic to C. Here the parallel hypersurface $d(z)=d_0$ is the unique minimal hypersurface of the family.
 \end{theorem}
 
 An outline of the paper is as follows. In section \ref{sec1} we derive the equations for a radial graph over a domain  $\Omega \subset \bfS^n$ which satisfies equation \ref{eq0.10} and is asymptotic to the cone over $\Gamma\subset \partial \Omega$. If we specialize to $\Gamma$ an isoparametric hypersurface our problem then reduces to a singular ode problem for $\varphi(d)$. We then reformulate the problem as a singular initial value problem for a nonlinear second order ordinary differential equation to  make it more tractable. Even so, the problem is quite nonstandard and requires a novel approximation to obtain a globally smooth solution. 
This is carried out in sections 3 and 4 where we prove existence and  uniqueness of a smooth solution. In section 5 we prove Gevrey 2 regularity for $g(d)=e^{-2\varphi(d)}$ or equivalently its inverse function $\gamma(s)$. The proof of Theorem \ref{th0.1} follows from these results as sketched at the end of section 4.  Section 6 contains several Fa\`{a} di Bruno variants and lemmas used in the proof of Gevrey regularity.

\section{The self-shrinker equation for a radial graph asymptotic to a mean convex cone}
\label{sec1}
\setcounter{equation}{0}
Let $\Gamma$ be a compact embedded hypersurface in the upper hemisphere of the unit sphere $\bfS^n\subset \bfR^{n+1}$ which is strictly mean convex with respect to the
inner orientation and let $C$ be the cone over $\Gamma$. Let $A$ be an annular neighborhood of the outside of $\Gamma$ and let $S=\{e^{v(z)} z: z\in A\}$ be a radial graph over A that is asymptotic to C. Notice that to satisfy this last condition we need both $v(z)\goto \infty$ and $d(z)e^{v(z)}\goto 0$ as $d(z)\goto 0$ where $d(z)>0$ is the distance function to $\Gamma$ in $A$.

We let $\sigma_{ij}$ denote the metric of $\bfS^n$ with respect to local coordinates with 
$(\sigma^{ij})=(\sigma_{ij})^{-1}$. We use the shorthand $v^i=\sigma^{ik}v_k,\, v_{ij}=\nabla_i \nabla_j v$ for covariant derivative with respect to the metric $\sigma_{ij}$.
Then \cite{GS} the  outward unit normal to $S$ is $\nu=\frac{z-\nabla v}{w},\, w=\sqrt{1+|\nabla v|^2}$ and the mean curvature with respect to the orientation $\nu$ is given by
\be \label{eq1.10}
H= \frac{a^{ij}v_{ij}-n}{e^v w}
\ee
where $a^{ij}:=\sigma^{ij}-\frac{v^i v^j}{w^2}$. The self-shrinker equation $H=-\frac12 X\cdot \nu$ is then 
\[\frac{a^{ij}v_{ij}-n}{e^v w}= e^v w~,\]
or 
\be \label{eq1.20}
a^{ij}v_{ij}=n-\frac12 e^{2v}
\ee
with  singular boundary condition $v(z)\goto \infty$ as $z\goto \Gamma$. It is not too difficult to see that the singularity of $v(z)$ must be asymptotically $\log{\frac{H}d}$. 

This is a seemingly ill-posed problem since by the uniqueness theorem of Lu Wang \cite{W1} we cannot hope to impose boundary conditions on $\partial A\setminus \Gamma$ so we are stuck with a  Cauchy problem for a singular equation that  resembles a backward heat equation. In this paper we  simplify the problem by reducing it to a singular ordinary differential equation. In order to accomplish this we assume that $\Gamma$ is part of an {\em isoparametric} family of parallel hypersurfaces $\Gamma_d$ with constant principle curvatures. More precisely if $f:\Gamma\goto \bfS^n$, then  $\Gamma_d=f_d(\Gamma)$ where $f_d(x):=\cos{d}\, f(z)-\sin{d}\, N(z)$
and $N(z)$ is chosen with the outer orientation so we are considering only the outer part of the family relative to $\Gamma$.
Except for the unique minimal hypersurface of the family, each  parallel hypersurface has (if properly oriented) positive constant mean curvature.   If  we denote by $H(d)$  the mean curvature of the parallel hypersurface  $\Gamma_d$ at distance d with orientation consistent with $\Gamma$), then
\[ H(d)>0,\, H'(d)<0,\, 0<d<d_1, \, H(d_1)=0\]
 and we may look for a solution of the form $v=\varphi(d(z))$. Substitution of this ansatz into \eqref{eq1.20} gives
\[(\sigma^{ij}- \frac{\varphi'^2 d^i d^j}{1+\varphi'^2})(\varphi' d_{ij}+\varphi'' d_i d_j)=n-\frac12 e^{2\varphi}.\]
Since $|\nabla d|^2=\sigma^{ij}d_i d_j=1,\, \sum_i d_i d_{ij}=0$ and $ \sigma^{ij}d_{ij}=\Delta d=H(d)$, we arrive
after simplification to the equation
\be \label{eq1.30}
\frac{\varphi''}{1+\varphi'^2}+\varphi' H(d)=n-\frac12 e^{2\varphi}.
\ee

Because of the expected form of the singularity of $f$, we change variables by setting 
$\varphi:=-\frac12 \log{g(d)}$.  Then \eqref{eq1.30} transforms to
\be \label{eq1.40}
\frac1{2g}(1-H(d) g')+2\ \frac{g'^2-gg''}{g'^2+4g^2}=n
\ee
with boundary condition $g(0)=0$ and the implied additional condition (assuming we can find a smooth enough solution) $g'(0)=H(0)>0$.  To simplify our analysis in the next section, we set $s=g(d)$ with inverse $d=\gamma(s)$. Then after simplification, we
arrive at the singular initial value problem
\be \label{eq1.50}
(1-2ns)\gamma'(s)-H(\gamma)+4\frac{ s\gamma'(s)+  s^2 \gamma''(s) }{1+4s^2\gamma'^2(s)}               =0,\,\,\gamma(0)=0~, \,\, s>0
\ee
and we seek a smooth solution on a uniform interval $[0,s_0]$ with $0<\gamma(s)<d_0$
and $\gamma'(s)>0$. Formally it is not hard to see that all the derivatives  $g^{(k)}(0)$  are determined. That is, 
there exists a formal  power series  solution of \eqref{eq1.50}
\be \label{eq1.60}
\sum_{k=1}^{\infty}\frac{A_k}{k!}s^k
\ee
since it is straightforward (since all the compositions in \eqref{eq1.60} are well defined) to see that $A_{k+1}$ is recursively determined from
$A_1.\ldots,A_k$. In fact,
\be \label{eq1.70}
A_{k+1}=(-4k^2+(2n+H'(0))k)A_k+P(A_1,\ldots,A_{k-1})
\ee
where $P$ is an explicit polynomial.
This follows from the discussion in section \ref{App}.

\section{Approximation and apriori estimates}
\label{sec2}
\setcounter{equation}{0}

In this section  we study a special $\e$-regularized initial value problem 
\be \label{eq2.10}
\begin{aligned}
&\e \gamma''(s)+4\frac{s^2 \gamma''(s)+s\gamma'(s)}{1+4s^2\gamma'^2(s)}   +(1-2ns)\gamma'(s)-H(\gamma) =0,\,\, s>0 \\
&\gamma(0)=0,\, \gamma'(0)= B(\e):=\sum_{i=0}^N B_i \e^i, \, B_0=H(0)
\end{aligned}
\ee
for $\gamma=\gamma(s,\e)$, where $\e>0$ is small, $N$ is a fixed large integer.  The $B_i$ will be determined recursively as polynomial functions of the $A_k$
so that $\gamma^{(k+2)}(0)$ is bounded in terms of  $N$ but independent of $\e$
and $\lim_{\e \goto 0}\gamma^{(k)}(0)=A_k,\, k=1,\ldots, N+1$, where the $A_k$ are the coefficients in the formal power series solution
\eqref{eq1.60}. For any choice of $B_1,\ldots,B_N$,   there is an analytic solution $\gamma(s)=\gamma(s;\e)$ of the initial value problem \eqref{eq2.10} on some small interval $0<s<s_0(\e)$. 

\begin{proposition} \label{prop2.1} There exists $B_1,\ldots,B_N$  such that 
\be \label{eq2.105}
\gamma^{(k)}(0)=A_k+((-1)^{k-1} B_{k}+\phi_{k}(B_1,\ldots,B_{k-1}))\e+O(\e^2),\, k=1,\ldots, N+1
\ee
Moreover $|B_k|,\,|\gamma^{(k)}(0)|\leq C=C(H,N),\, k=1,\ldots,N+1$ where $C$ is independent of $\e$.
\end{proposition}
\begin{proof}
We differentiate \eqref{eq2.10} k times 
%\be \label{eq2.11} \e \gamma^{(k+2)} +(\frac{d}{ds})^k \{2s\frac{d}{ds}\arctan{(2s\gamma'(s))}
%+(1-2ns)\gamma'(s)-H(\gamma)\}=0.
%\ee
making use of the identities
\[(\frac{d}{ds})^l (s\frac{d}{ds})^p u(s)=(s\frac{d}{ds}+l)^p u^{(l)}(s),\]
\[4\frac{s^2 \gamma''(s)+s\gamma'(s)}{1+4s^2\gamma'^2(s)} =2s\frac{d}{ds}\arctan{(2s\gamma'(s))}~.\]
Thus
\be \label{eq2.12}
\begin{aligned}
&(\frac{d}{ds})^k[(1-2ns)\gamma'(s)](0)=\gamma^{(k+1)}(0)-2nk\gamma^{(k)}(0),\\
&(\frac{d}{ds})^k (2s\frac{d}{ds}\arctan{(2s\gamma'(s))})(0)=2k (\frac{d}{ds})^k\arctan{(2s\gamma'(s))})(0),
\end{aligned}
\ee
so that
\be \label{eq2.16} 
\begin{aligned}
\e \gamma^{(k+2)}(0)&=-\gamma^{(k+1)}(0)+2nk\gamma^{(k)}(0)+(\frac{d}{ds})^k 
\{H(\gamma)-2k\arctan{(2s\gamma'(s))}\}(0)\\
&:=-\gamma^{(k+1)}(0)+A^{\e}_{k+1}
\end{aligned}
\ee
where the last term on the  right hand side of \eqref{eq2.16} contains derivatives of $H$ and $\gamma$ evaluated at $0$ of 
order at most $k$. A more explicit expression for the right hand side can be obtained using the formulas \eqref{eq40}, \eqref{eq2.15} of section 6.
Note that the coefficients $A_k$ of the formal solution of \eqref{eq1.60} are obtained
by setting $\e=0$ and $A_1=\gamma'(0)=H(0)$. 
This leads to the simple but important observation:
 \be \label{eq2.165}
 A_{k+1}=F(A_1,\ldots,A_k):=  A^{\e}_{k+1}|_{\e=0}~.
 \ee

It is straightforward to verify that
\be \label{eq2.17}
|A_k|\leq C^{k}(H) (k!)^2~;
\ee
in other words the formal solution is in the Gevrey class $\mathcal{G}^2$ (of order 2) and in fact no better.
Similarly for the $\e$ regularized problem with initial condition $\gamma'(0)=B(\e)$,
the derivatives $\gamma^{(k)}(0)$ are recursively well defined by \eqref{eq2.16}.
Note  that $A^{\e}_{k+1}$ is a polynomial in $\e$ with coefficients depending on the $\{B_i\}$.  By our choice $B_0=H(0),\, \gamma''(0)= -\sum_{i=1}^N B_i \e^{i-1}$ and we choose $B_1=-A_2$. Suppose we have chosen $B_0=H(0),\, B_1=-A_2,\ldots, B_{k}$ so that
\[\gamma^{(l)}(0)=A_l+((-1)^{l-1} B_{l}+\phi_{l}(B_1,\ldots,B_{l-1}))\e+O(\e^2), \,\,l=2,\ldots,k+1~.\]
Claim: $\gamma^{(k+2)}(0)=(-1)^{k+1} B_{k+1}+\phi(B_1,\ldots, B_{k}) +O(\e)$. \\
Assuming the claim, we complete the induction
by defining $B_{k+1}$ by the relation 
\[A_{k+2}=(-1)^{k+1} B_{k+1}+\phi(B_1,\ldots, B_{k}) ~. \]
 To prove the claim we will need a more explicit expression for $A^{\e}_{k+1}$
so that we can expand $\e\gamma^{k+2}(0)$ in powers of $\e$.

 By \eqref{eq2.16} and \eqref{eq2.165} and our induction hypothesis, the constant term in the expansion of $\e \gamma^{(k+2)}$ in powers of $\e$ is automatically $0$ and the $\e$ term begins with $(-1)^{k+1} B_k+\phi_{k+1}(B_1,\ldots,B_{k-1})$ coming from 
$\gamma^{(k+1)}(0)$.
It remains to  verify that the $\e$ terms of $A^{\e}_{k+1}$ depends only on
 $B_1,\ldots,B_{k-1}$. But this is now obvious from the formulas \eqref{eq40}, \eqref{eq2.15} and our induction hypothesis. Thus
 \be \label{eq2.18}
 \begin{aligned}
 &\e\gamma^{k+2}(0)=(-A_{k+1}+A^{\e}_{k+1}|_{\e=0})+((-1)^{k+1}B_k+\phi(B_1,\ldots, B_{k-1}))\e+O(\e^2)\\
 &\gamma^{k+2}(0)=(-1)^{k+1}B_k+\phi(B_1,\ldots, B_{k-1})+O(\e)~,
 \end{aligned}
 \ee
 completing the proof.
\end{proof}

Now that we have control of the first $N+1$  derivatives of $\gamma$ at the origin, we will inductively derive energy estimates that imply
$0<\gamma(s)< d_0, \gamma'(s)>0,\, |\gamma^{(k)}(s)|\leq C_k,\, k=1,\ldots, N$ independent of $\e$ (as $\e \goto 0$) on a uniform interval $(0,s_0)$. We assume, based on Proposition \ref{prop2.1},  that $\e$ is chosen  so small that %\e=O(1/N^2)?
\be \label{eq2.19}
\frac12 H(0)\leq \gamma'(0) \leq 2H(0),\, , | \gamma^{(k)}(0)-A_k|\leq 1,\,
 \e | \gamma^{(k)}(0)|^2 \leq 1,\, k=1,\ldots, N+1~.
\ee

In the following discussion, $C$ will denote a constant depending on $H$ and dimension $n$ which may change from line to line.
We multiply \eqref{eq2.10} by $\gamma'(s)$ and integrate to obtain
\be \label{eq2.20}
\e\frac{\gamma'^2(s)-\gamma'^2(0)}2+\frac12\log{(1+4s^2\gamma'^2(s))} +  
\int_0^s(1-2nt)\gamma'^2(t)dt=\mcH(\gamma(s))
\ee
where $\mcH'(t)=H(t),\, \mcH(0)=0$. We restrict $0\leq s\leq \frac1{4n}$.  Since
$\mcH(\gamma(s))\leq C\gamma(s)$ we conclude that

\[ \gamma^2(s)=(\int_0^s \gamma'(t) dt)^2  \leq  s\int_0^s \gamma'^2(t)dt \leq   
s(\e \gamma'^2(0)+2C\gamma(s))~.\]
Hence $\gamma(s)\leq  2C s + \gamma'(0) \sqrt{ \e s} \leq C(s+\sqrt{ \e s})$.
Thus by 
\eqref{eq2.20} and \eqref{eq2.10}, we  obtain the preliminary estimates
\begin{lemma} \label{lem2.1} 
\be  \label{eq2.30}
\begin{aligned}
&\gamma(s)\leq C(s+\sqrt{ \e s})\\
&0\leq  s\gamma'(s)\leq  \max{(\gamma'(0) s, C (s + \sqrt{ \e s})^{\frac12}))} \leq C (s + \sqrt{ \e s})^{\frac12},\\
&s^2|\gamma''(s)|\leq \gamma'(s)+C,\, \, \int_0^s (\gamma'(t))^2 dt\leq C
\end{aligned}
\ee
\end{lemma}

 We differentiate \eqref{eq2.10} $k\geq 1$ times making use of the identity
\[(\frac{d}{ds})^l (s\frac{d}{ds})^p u(s)=(s\frac{d}{ds}+l)^p u^{(l)}(s)\]
to obtain
\be \label{eq2.40}
\begin{aligned}
\e\gamma^{(k+2)}+& 4\sum_{l=0}^k {k\choose l} \left(s^2 \gamma^{(l+2)}+(2l+1)s \gamma^{(l+1)}+l^2 \gamma^{(l)}\right)(\frac{d}{ds})^{k-l}\eta(s\gamma')\\
+&(1-2ns)\gamma^{(k+1)}-2nk\gamma^{(k)}-(\frac{d}{ds})^k H(\gamma)=0~,
\end{aligned}
\ee
where $\eta(x)=\frac1{1+4x^2}$. An explicit formula for $(\frac{d}{ds})^{k-l}\eta(s\gamma')$ is given in formula \eqref{eq35} of the Appendix.
We multiply \eqref{eq2.40} by $\gamma^{(k+1)}$ and integrate isolating  the crucial terms involving $(\gamma^{(k+1)})^2$:
\be \label{eq2.50}
\begin{aligned}
\hspace{.5in} &\frac{\e}2[(\gamma^{(k+1)})^2(s)-(\gamma^{(k+1)})^2(0)]+2s^2\eta(s\gamma')(\gamma^{(k+1)})^2(s)\\
&+4\sum_{l=1}^{k-2}{k\choose l}\int_0^s\gamma^{(k+1)}(t)[t^2\gamma^{(l+2)}+(2l+1)t \gamma^{(l+1)}+l^2 \gamma^{(l)}]
(\frac{d}{dt})^{k-l}\eta(t\gamma')~dt\\
&+\int_0^s [-2\frac{d}{dt}(t^2\eta(t\gamma'))+4(2k+1)t\eta(t\gamma')
  + 4k t^2\frac{d}{dt}\eta(t\gamma')](\gamma^{(k+1)})^2~dt  \\
  &+4\int_0^s \gamma^{(k+1)}(t^2 \gamma''+t\gamma')\{ (\frac{d}{dt})^k\eta(t\gamma') 
  +8t\gamma' \eta^2(t\gamma')(t\gamma^{(k+1)}+k\gamma^{(k)})\}~dt\\
  &-32\int_0^s [t^2 \gamma' (t^2\gamma''+t\gamma')]\eta^2(t\gamma') (\gamma^{(k+1)})^2~dt
 -32k\int_0^s[t\gamma'( t^2\gamma''+t\gamma') ]\eta^2 \gamma^{(k)}\gamma^{(k+1)}~dt  \\
&\int_0^s (1-2nt)(\gamma^{(k+1)})^2(t)~dt
+\int_0^s [4k^2 \eta(t\gamma')+ 4(2k-1)t\frac{d}{dt}\eta(t\gamma')-2nk]\gamma^{(k+1)}\gamma^{(k)}~dt\\
&+4k(k-1)^2 \int_0^s \gamma^{(k+1)} \gamma^{(k-1)}\frac{d}{dt}\eta(t\gamma')~dt 
-\int_0^s (\frac{d}{dt})^k H(\gamma)\gamma^{(k+1)}~dt  =0 .
\end{aligned}
\ee
where we have used \eqref{eq35} with $m=k$ to rewrite the term $4\int_0^s \gamma^{(k+1)} (t^2 \gamma''+t\gamma')( \frac{d}{dt})^k\eta(t\gamma' )~dt$.\\

For k=1 using the preliminary estimates  of Lemma \ref{lem2.1} we obtain 
$I \geq \int_0^s(1-C(t+\sqrt{\e t}))(\gamma'')^2~dt$ and therefore we obtain for $k=1$
\be \label{eq2.70}
\frac{\e}2(\gamma''(s)^2-\gamma''(0)^2)+2s^2\eta(s\gamma')\gamma''^2(s)+I-
\int_0^s[8\eta(t\gamma')+H'(\gamma)]\gamma' \gamma''~dt=0
\ee
Using the estimate \eqref{eq2.20} in \eqref{eq2.70} yields for $0\leq s\leq s_0$ ($s_0$ a uniform constant)
\be \label{eq2.80}
\frac{\e}2\gamma''(s)^2 +\frac{2s^2 \gamma''^2(s)}{1+4(s\gamma')^2}+
\int_0^s (\gamma'')^2~dt\leq	 C
\ee

In particular we have the improved estimates
\be \label{eq2.90}
|\gamma'(s)-\gamma'(0)|\leq C\sqrt{s}, \,\, s|\gamma''(s)|\leq C
\ee
and using this in Lemma \ref{lem2.1} yields
\be \label{eq2.91}
s^2|\gamma''(s)| \leq s^2 |\gamma''(0)|+C\sqrt{s}.
\ee

Collecting all terms $I$   in \eqref{eq2.50} involving $\int_0^s (\cdot)(\gamma^{(k+1)})^2 ~dt $, we obtain
\be \label{eq2.60}
I=\int_0^s\{(1-2nt)+8kt\eta(t\gamma')-[(32(k+1)t^2\gamma'(t)(t^2\gamma''+
\gamma' ]\eta^2(t\gamma')\}(\gamma^{(k+1)})^2~dt.
\ee

Inserting \eqref{eq2.90}, \eqref{eq2.91}  in \eqref{eq2.60} yields
\be \label{eq2.100}
  I= \int_0^s (1+o(1))(\gamma^{(k+1)})^2~dt,\,\,\, k\geq 2
\ee
uniformly in $k$.

\begin{theorem}\label{the} Let $\gamma=\gamma(s;\e,N)$ be a solution of the $\e$-regularized initial value problem \eqref{eq2.10} on the interval $[0,s]$.
There exists $s_0$ sufficiently small   independent of $\e$ and $N$ so that 
\be \label{eq2.110}
\|\gamma^{(j)}||_{L^2[0,s]}+|(\frac{d}{ds})^{(j-1)}(s\gamma'(s))|~dt \leq C(N),\, 1\leq j\leq N+1,\, s\in[0,s_0].
\ee

\end{theorem}
\begin{proof} 
We will prove \eqref{eq2.110} by induction. We have already proved the estimates \eqref{eq2.110} for 
$j=1,2$, so $s\gamma''+\gamma'(s)$ and $\frac{d}{ds}\eta(s\gamma')$ are 
 bounded. Moreover the crucial  \eqref{eq2.100} holds.
 In the following discussion $C(N)$ will be a generic constant independent of $\e$ and $s_0$ which may change from line to line. Suppose that  \eqref{eq2.110} holds for $j=1,\ldots, k$. Note that  
 \[s^2 \gamma^{(l+2)}+(2l+1)s \gamma^{(l+1)}+l^2 \gamma^{(l)}=s(\frac{d}{ds})^{(l+1)}(s\gamma'(s))+l
 (\frac{d}{ds})^{(l)}(s\gamma'(s))\]
 is uniformly bounded for $1\leq l \leq k-2$ and so is $(\frac{d}{ds})^{k-l}\eta(s\gamma')$ by
 the Fa\`{a} di Bruno formula \eqref{eq35} of the Appendix.  Note also that \eqref{eq35} implies
 \[(\frac{d}{ds})^k \eta(s\gamma'(s)=-8s\gamma'(s)\eta^2(s\gamma^{(k+1)}(s)+k\gamma^{(k)}(s))+\text{bounded terms}~.\]
 Hence the term  
 \be \begin{aligned} 
 &4\int_0^s \gamma^{(k+1)}(t^2 \gamma''+t\gamma')\{ (\frac{d}{dt})^k\eta(t\gamma') 
  +8(t\gamma') \eta^2(t\gamma')(t\gamma^{(k+1)}+k\gamma^{(k)}\}~dt\\
&=O(C(N)\int_0^s |\gamma^{(k+1)}|~dt)
=O(\frac1{10}\int_0^s(\gamma^{(k+1)})^2~dt+C(N))~.
 \end{aligned} \ee
 Moreover  the terms involving $\int_0^s (\cdot)\gamma^{(k+1)}\gamma^{(j)}dt,\, j=k \,\,\text{or $j=k-1$}$ satisfy the same bounds.
 Similarly, $(\frac{d}{ds})^k=H'(\gamma))\gamma^{(k)}+O(C(N)$  so
 \[|\int_0^s (\frac{d}{dt})^k H(\gamma)\gamma^{(k+1)}~dt | \leq \frac1{10}\int_0^s(\gamma^{(k+1)})^2~dt+C(N)~.\]
 Thus for $s_0$ sufficiently small we derive from \eqref{eq2.50}
 \[\|\gamma^{(k+1)}||_{L^2[0,s]}+|(\frac{d}{ds})^{(k)}(s\gamma'(s))|~dt \leq C(N),\,s\in[0,s_0]~,\]
 and the induction is complete.

\end{proof}
\section{Existence and uniqueness}
\label{sec3}
\setcounter{equation}{0}
In this section we apply the results of Section \ref{sec2} to prove the existence of a $C^{\infty}[0,s_0]$ solution of the singular initial value problem \eqref{eq1.50}. We start by proving the uniqueness of solutions to both the $\e$ regularized problem \eqref{eq2.10} and the original singular initial value problem \eqref{eq1.50}.

Let $\gamma_1(s),\,\gamma_2(s)$ be two solutions of \eqref{eq2.10} and set $u(s)=\gamma_1(s)-\gamma_2(s),\, \gamma_{\theta}(s):=\gamma_2(s)+\theta u(s)$. Then using the mean value theorem, we see that
\be \label{eq3.10}
\e u''(s)+(1-2ns)u'(s)-b(s)u(s)+4s\frac{d}{ds}(2sa(s)u'(s))=0
\ee
where
\be \label{eq3.20}
b(s):=\int_0^1 H'(\gamma_{\theta}(s))d\theta,\,\,a(s):=\int_0^1\frac{d\theta}{1+4s^2(\gamma_{\theta}'(s))^2}~.
\ee
Multiplying \eqref{eq3.10} by $u'$ and integrating by parts gives
\be \label{eq3.30}
\begin{aligned}
&\frac{\e}2(u'^2(s)-u'^2(0))+\int_0^s(1-2nt)u'^2(t)~dt+\frac12(b(0)u^2(0)-b(s)u^2(s))\\  &+4a(s)s^2u'^2(s)+4\int_0^s t^2 a'(t)u'^2(t)~dt=0
\end{aligned}
\ee

As a consequence of \eqref{eq3.30} we have the following uniqueness theorem.\begin{theorem} \label{th3.1}
Assume $H\in C^1[0,d_0),\, H'(d)<0$. Then there is at most one $C^2[0,s_0)$ solution to the $\e$ regularized initial value problem \eqref{eq2.10} and the singular initial value problem \eqref{eq1.50}.
\end{theorem}
\begin{proof} In both cases, we have $u(0)=u'(0)=0$ and $b(s)<0, \, a(s)>0$ and $a'(s)=O(s)$. Hence $(1-2nt+4t^2 a'(t))>0$ for $0<t<s$ small so the left hand side of \eqref{eq3.30} is nonnegative. Hence
 $u(s)\equiv 0$  for $s>0$ small and the theorem follows.
\end{proof}

\begin{theorem} \label{th3.2}There exists $s_0$ depending only on $H$ such that the initial value problem \eqref{eq1.50} has a unique solution $\gamma\in C^{\infty}[0,s_0]$. Moreover, $\gamma^{(k)}(0)=A_k,\,k=1,2,\ldots$.
\end{theorem}
\begin{proof} According to Theorems  ,  for any integer N, there is an analytic solution $\gamma_N(s;\e) $ of the $\e$ regularized initial value problem \eqref{eq2.10} on a uniform interval $[0,s_0)$ independent of $\e$ and N. Moreover, 
\[\Vert{\gamma_N}\Vert _{C^{N+1}[0,s_0]} \leq C(N),\, \, \lim_{\e \goto 0}\gamma_N^{(k)}(0,\e)=A_k~.\]
Hence taking limits and using Theorem \ref{th3.1}, there is a solution $\gamma(s)$ of \eqref{eq1.50} with $\gamma^{(k)}(0)=A_k,\,k=1,2,\ldots, N$. Since N is arbitrary, the theorem follows.
\end{proof}

{\bf Proof of Theorem \ref{th0.1}}: According to Theorem \ref{th3.1} and the calculations of
section \ref{sec2}, there exists a unique solution $\varphi(d)=
-\frac12\log{g(d)}$ to \eqref{eq1.30}
on a small interval $0<d<d_1$.  That is the radial graph 
$\Sigma=\{e^{\varphi(d(z))}: z\in A\}$, where $d(z)$ is the distance function to $\Gamma$ in $A=\{z: 0<d(z)\leq d_1+\e\}$  is an end of a  self-shrinker to the mean curvature  flow. Moreover according to Theorem \ref{th4.1} of section \ref{sec4}, $g(d)=e^{-2\varphi(d)}$ is in the Gevrey class $\mathcal{G}^2$, since $g(d)$ is the inverse function of $\gamma(s)$. It remains to show that $\varphi(d)$ exists on the interval $(0,d_0+\e)$ for some small $\e>0$ where the parallel hypersurface to $\Gamma$, $d(z)=d_0$ is the 
unique minimal hypersurface of the family. To see this we multiply \eqref{eq1.30} by $2\varphi'(d)$ and integrate from $\frac{d_1}2$ to $d$ to obtain
\be \label{eq3.50}
\log{(1+\varphi'^2(d))}+2\int_{\frac{d_1}2}^d H(t)\varphi'^2(t)dt=2n\varphi(d)-\frac12e^{2\varphi(d)}+C(d_1)~.
\ee
Note that the right hand side of \eqref{eq3.50} tends to negative infinity as $\varphi \goto \pm \infty$ while the left hand side remains strictly positive while $H\geq 0$. It follows easily that the solution continues past $d=d_0$, completing the proof.

\section{Further Gevrey regularity of the solution}
\label{sec4}
In this section we show that the solution of the singular initial value problem \eqref{eq1.50} is in the Gevrey class $\mathcal{G}^2[0,s_0]$ for a uniform constant $s_0$.
For the following calculations we will need to assume that $\|\gamma''\|_{L^2[0,s]}+|\frac{d}{ds}(s\gamma'(s))|$ is small on $[0,s_0]$. We can achieve this by the  rescaling $\tilde{\gamma}(s)=\gamma(\lambda s),\, 0\leq s \leq \frac{s_0}{\lambda}$ with $\lambda$ small. For $\frac{d}{ds}[s\tilde{\gamma}'](s)=\lambda \frac{d}{ds}[s\gamma'](\lambda s),\,\, \|\tilde{\gamma}''\|_{L^2[0,\frac{s_0}{\lambda}]}=\lambda^{\frac32} \|\gamma''\|_{L^2[0,s_0]}$
 and $\tilde{\gamma}(s)$ satisfies
\be \label{eq4.10}
\begin{aligned}
& 4\frac{s^2 \tilde{\gamma}''(s)+s\tilde{\gamma}'(s)}{1+4s^2\tilde{\gamma}'^2(s)}   +(\frac1{\lambda}-2ns)\tilde{\gamma}'(s)-H(\tilde{\gamma}) =0,\,\, 0\leq s \leq \frac{s_0}{\lambda} \\
&\tilde{\gamma}(0)=0~.
\end{aligned}
\ee
The only difference between \eqref{eq4.10} and the original equation
\eqref{eq1.50} is in the coefficient $(\frac1{\lambda}-2ns)$ instead of $1-2ns$ which only improves the estimates.
Thus there is no loss of generality in working with \eqref{eq1.50} and assuming the necessary smallness conditions.\\

\begin{theorem}\label{th4.1} There exists $s_0>0$ small  and $M$ large (depending on $H$) so that if $u\in C^{\infty}[0,s_0]$ is the unique solution of \eqref{eq1.50}, then \be \label{eq4.20} \|\gamma^{(j)}\|_{L^2[0,s]}+|(\frac{d}{ds})^{j-1}(s\gamma'(s))|\leq  M^{j-3}\frac{(j-2)!^2}{(j-1)},\, \text{on $[0,s_0]$}.
\ee
for all integers $j\geq 2$.
\end{theorem}

 We will prove \eqref{eq1.50} by induction assuming that it holds for $j=2,\ldots, k$. Note that for $M$ to be chosen later depending only on $H$, the starting case $j=2$ can be achieved by our smallness assumptions.

\begin{proposition}\label{prop4.1} Let $k\geq 2$ and assume the induction hypothesis \eqref{eq4.20} for $j=2,\ldots, k$. Then
\be \begin{aligned}
& \text{ i. for $1\leq n \leq k-1$,}\,\,
|(\frac{d}{ds})^n(\eta(s\gamma'(s))| \leq   16M^{n-2}\frac{(n-1)!^2}{n}
         ~,\\
& \text{ii.}\,\, |(\frac{d}{ds})^k(\eta(s\gamma'(s))+8s\gamma'(s)\eta^2[\gamma^{(k+1)}(s)+k\gamma^{(k)}(s)]|\leq 
C M^{k-4}\frac{k!(k-2)!}{(k-1)^2}~,
\end{aligned}
\ee
\end{proposition}
\begin{proof} To prove part i. we use the alternate version of the Fa\`{a} di Bruno formula from Lemma \ref{lemA2}:
\[(\frac{d}{ds})^n(\eta(s\gamma'(s))=\sum_{l=1}^n {n \choose l}(-2)^l \Lambda_{l+1}(s\gamma'(s))\left\{ (\frac{d}{dh})^{n-l}[\int_0^1 \gamma''(s+h\theta)~d\theta +\gamma'(s+h)]^l\right\}_{h=0}~.\]
Under the induction hypothesis \eqref{eq1.50}, Lemma \ref{lemA4} implies (we assume $\frac{C}{M}\leq \frac12$ below)
\be \begin{aligned} \label{eq4.40}
&|(\frac{d}{ds})^n(\eta(s\gamma'(s))| \leq \sum_{l=1}^n{n\choose l}2^l l!(\frac{C}{M})^{l-1}M^{n-l-1}\frac{(n-l)!^2}{(n-l+1)^2}\\
&=\frac{M^n}{C} n!\sum_{l=1}^n(\frac{2C}{M^2})^l \frac{(n-l)!}{(n-l+1)^2}\leq16M^{n-2}\frac{(n-1)!^2}{n}
\end{aligned} \ee
The proof of part ii. is essentially the same.
\end{proof}

\begin{lemma}\label{lem4.1}Let $k\geq 2$ and assume the induction hypothesis \eqref{eq4.20} for $j=2,\ldots, k$. Then \be \label{eq4.50}
\begin{aligned}
 &i.\,\, \|(\frac{d}{ds})^k[H(\gamma(s))]\|_{L^2[0,s]}\leq C_1M^{k-3}\frac{(k-2)!^2}{k-1}+\sqrt{s_0} C_1^2C M^{k-4}(k-2)!^2~,\\
 &ii.\,\,\hspace{.1in} |\int_0^s (\frac{d}{dt})^k H(\gamma)\gamma^{(k+1)}~dt |\leq
 C_2 M^{k-3}(k-2)!^2 \|\gamma^{(k+1)}\|_{L^2[0,s]}~,
 \end{aligned}
 \ee
where $C$ is a universal constant and $C_1=C_1(H),\, C_2=C_2(H)$.
\end{lemma}
\begin{proof} According to Corollary \ref{corA1}
\[(\frac{d}{ds})^mH(\gamma(s))=H'(\gamma(s))\gamma^{(m)}(s)+\sum_{l=2}^m {m \choose l} H^{(l)}(g(s))\left\{(\frac{d}{dh})^{m-l}[\int_0^1\gamma'(s+h\theta)~d\theta]^l\right\}_{h=0}~.\]
Moreover by the induction hypothesis $u(h;s):= \int_0^1\gamma'(s+h\theta)~d\theta$ satisfies
\[ |(\frac{d}{dh})^j[u(h;s)]| \leq \frac{1}{j+1}\sqrt{s-h}\|\gamma^{(j+2)}\|_{L^2[0,s-h]}\leq \sqrt{s_0-h}M^{j-1}\frac{j!^2}{(j+1)^2}\]
for $0\leq j \leq k-2$. Hence we can use Lemma \ref{lemA3} to conclude as in Proposition \ref{prop4.1} ii. that
part i. of \eqref{eq4.50} holds. Part ii. follows immediately from part i.
\end{proof}

We now are in a position to complete the induction. From \eqref{eq2.50} with $\e=0$ we have
\be \label{eq4.60} 
\begin{aligned}
\hspace{.2in}&2s^2\eta(s\gamma')(\gamma^{(k+1)})^2(s)+\int_0^s(1+o(1))(\gamma^{(k+1)})^2~dt\\
&+4\sum_{l=1}^{k-2}{k\choose l}\int_0^s\gamma^{(k+1)}(t)[t^2\gamma^{(l+2)}+(2l+1)t \gamma^{(l+1)}+l^2 \gamma^{(l)}]
(\frac{d}{dt})^{k-l}\eta(t\gamma')~dt\\
  &+4\int_0^s \gamma^{(k+1)}(t^2 \gamma''+t\gamma')\{ (\frac{d}{dt})^k\eta(t\gamma') 
  +8t\gamma' \eta^2(t\gamma')(t\gamma^{(k+1)}+k\gamma^{(k)})\}~dt\\
  & -32k\int_0^s[t\gamma'( t^2\gamma''+t\gamma') ]\eta^2 \gamma^{(k)}\gamma^{(k+1)}~dt \\ 
&+\int_0^s [4k^2 \eta(t\gamma')+ 4(2k-1)t\frac{d}{dt}\eta(t\gamma')-2nk]\gamma^{(k+1)}\gamma^{(k)}~dt\\
&+4k(k-1)^2 \int_0^s \gamma^{(k+1)} \gamma^{(k-1)}\frac{d}{dt}\eta(t\gamma')~dt 
-\int_0^s (\frac{d}{dt})^k H(\gamma)\gamma^{(k+1)}~dt  =0 .
\end{aligned}
\ee

We now estimate the terms of \eqref{eq4.60} in order of difficulty. 
\begin{lemma}\label{lem4.2} Let $k\geq 2$ and assume the induction hypothesis \eqref{eq4.20} for $j=2,\ldots, k$.  Then
\be \label{eq4.70}
\begin{aligned} &\sum_{l=1}^{k-2}{k\choose l}\int_0^s\gamma^{(k+1)}(t)[t^2\gamma^{(l+2)}+(2l+1)t \gamma^{(l+1)}+l^2 \gamma^{(l)}] (\frac{d}{dt})^{k-l}\eta(t\gamma')~dt\\
 &\leq C \sqrt{s_0} \|\gamma^{(k+1)}\|_{L^2[0,s]} M^{k-2} \frac{k!(k-1)!}{(k+1)^2}~,
 \end{aligned}
\ee
where C is a universal constant.
\end{lemma}
\begin{proof} In the following we use the induction hypothesis and Proposition \ref{prop4.1} part i.
\be \label{eq4.80}
\begin{aligned}
&\sum_{l=1}^{k-2}{k\choose l}\int_0^s\gamma^{(k+1)}(t)[t^2\gamma^{(l+2)}+(2l+1)t \gamma^{(l+1)}+l^2 \gamma^{(l)}] (\frac{d}{dt})^{k-l}\eta(t\gamma')~dt\\
&=\sum_{l=1}^{k-2}{k\choose l}\int_0^s\gamma^{(k+1)}(t)[t(\frac{d}{dt})^{(l+1)}(t\gamma'(t))+l (\frac{d}{dt})^{(l)}(t\gamma'(t))](\frac{d}{dt})^{k-l}\eta(t\gamma')~dt\\
&\leq \sqrt{s_0} \|\gamma^{(k+1)}\|_{L^2[0,s]} \sum_{l=1}^{k-2} {k \choose l} \{s_0 M^{l-1}\frac{l!^2}{l+1}+M^{l-2}(l-1)!^2\}\cdot 16 M^{k-l-1}\frac{(k-l-1)!^2}{k-l}\\
&=\sqrt{s_0} \|\gamma^{(k+1)}\|_{L^2[0,s]}\cdot 16(s_0+\frac1M)  M^{k-2} k!\sum_{l=1}^{k-2}\frac{(l+1)!(k-l-1)!}{(l+1)^2 (k-l)^2}\\
&\leq 8 \sqrt{s_0} \|\gamma^{(k+1)}\|_{L^2[0,s]} M^{k-2} \frac{k!(k-1)!}{(k+1)^2}
\sum_{l=1}^{k-2}(\frac1{l+1}+\frac1{k-l-1})^2\\
&\leq 32 \sqrt{s_0} \|\gamma^{(k+1)}\|_{L^2[0,s]} M^{k-2} \frac{k!(k-1)!}{(k+1)^2}
\sum_{l=1}^{\infty}\frac1{l^2}\\
&=C \sqrt{s_0} \|\gamma^{(k+1)}\|_{L^2[0,s]} M^{k-2} \frac{k!(k-1)!}{(k+1)^2}
\end{aligned}
\ee
\end{proof}

\begin{lemma}\label{lem4.3} Let $k\geq 2$ and assume the induction hypothesis \eqref{eq4.20} for $j=2,\ldots, k$.  Then
\be \label{eq4.90}
\begin{aligned} &\int_0^s \gamma^{(k+1)}(t^2 \gamma''+t\gamma')\{ (\frac{d}{dt})^k\eta(t\gamma') 
  +8t\gamma' \eta^2(t\gamma')(t\gamma^{(k+1)}+k\gamma^{(k)})\}~dt\\
 &\leq C\sqrt{s_0}  \|\gamma^{(k+1)}\|_{L^2[0,s]}\,  M^{k-4}\frac{k!(k-2)!}{(k-1)^2}~,
 \end{aligned}
 \ee
 where $C=C(H)$.
\end{lemma}
\begin{proof} In the following we use the induction hypothesis and Proposition \ref{prop4.1} part ii.
\be \label{eq4.100}
\begin{aligned} 
&\int_0^s \gamma^{(k+1)}(t^2 \gamma''+t\gamma')\{ (\frac{d}{dt})^k\eta(t\gamma') 
  +8t\gamma' \eta^2(t\gamma')(t\gamma^{(k+1)}+k\gamma^{(k)})\}~dt\\
&\leq  \|\gamma^{(k+1)}\|_{L^2[0,s]}\, \sqrt{s_0}  \cdot C M^{k-4}\frac{k!(k-2)!}{(k-1)^2}\\
&=C\sqrt{s_0}  \|\gamma^{(k+1)}\|_{L^2[0,s]}\,  M^{k-4}\frac{k!(k-2)!}{(k-1)^2}\\
\end{aligned}
\ee
\end{proof}

Finally we estimate the remaining terms in \eqref{eq4.60}.
\begin{lemma}\label{lem4.4} Let $k\geq 2$ and assume the induction hypothesis \eqref{eq4.20} for $j=2,\ldots, k$.  Then for a constant $C=C(H)$,
\be \label{eq4.110}
\begin{aligned} &| -32k \int_0^s[t\gamma'( t^2\gamma''+t\gamma') ]\eta^2 \gamma^{(k)}\gamma^{(k+1)}~dt \\
&+\int_0^s [4k^2 \eta(t\gamma')+ 4(2k-1)t\frac{d}{dt}\eta(t\gamma')-2nk]\gamma^{(k+1)}\gamma^{(k)}~dt\\
&+4k(k-1)^2 \int_0^s \gamma^{(k+1)} \gamma^{(k-1)}\frac{d}{dt}\eta(t\gamma')~dt\, | \\
&\leq C  \|\gamma^{(k+1)}\|_{L^2[0,s]}( k^2\|\gamma^{(k)}\|_{L^2[0,s]}+ k^3\|\gamma^{(k-1)}\|_{L^2[0,s]})\\
& \leq C  \|\gamma^{(k+1)}\|_{L^2[0,s]}( M^{k-3}\frac{k^2 (k-2)!^2}{k-1}+ 
M^{k-4}\frac {k^3 (k-3)!^2}{k-2})\\
&=\frac{C}{M}\{ M^{k-2}\frac{ (k-1)!^2}{k} \|\gamma^{(k+1)}\|_{L^2[0,s]}\}~.
\end{aligned}
\ee
\end{lemma}

To complete the proof of Theorem \ref{th4.1} we need to show 
\be \label{eq4.120}\|\gamma^{(k+1)}\|_{L^2[0,s]}+|(\frac{d}{ds})^{k}(s\gamma'(s))|\leq  M^{k-2}\frac{(k-1)!^2}{(k)},\, \text{on $[0,s_0]$}.
\ee
Using Proposition \ref{prop4.1} and Lemmas \ref{lem4.1}, \ref{lem4.2}, \ref{lem4.3} and \ref{lem4.4} to estimate the error terms in \eqref{eq4.60}, we find
\be \label{eq4.130}
\begin{aligned} &\|\gamma^{(k+1)}\|_{L^2[0,s]}+|(\frac{d}{ds})^{k}(s\gamma'(s))|\\
&\leq C\sqrt{s_0}\{  M^{k-2} \frac{k!(k-1)!}{(k+1)^2}+  M^{k-4}\frac{k!(k-2)!}{(k-1)^2}\} \\
&+\frac{C}{M} M^{k-2}\frac{ (k-1)!^2}{k}+ C_2 M^{k-3}(k-2)!^2 \\
&\leq M^{k-2}\frac{ (k-1)!^2}{k}\{C\sqrt{s_0}+ C\sqrt{s_0}\frac{1}{M^2}\}\\
&\leq C\sqrt{s_0} M^{k-2}\frac{ (k-1)!^2}{k}\leq  M^{k-2}\frac{ (k-1)!^2}{k}
\end{aligned}
\ee
for $s_0$ small and $M$ large. This completes the proof of Theorem \ref{th4.1}.

\section{Appendix: Fa\`{a} di Bruno formulas}
\label{App}
\setcounter{equation}{0}

In this section we recall the Fa\`{a} di Bruno formulas for the higher derivatives of composite functions of one variable. An excellent reference is the interesting and informative survey article of W. P. Johnson \cite{J}.
The usual version of the formula is
\be \label{eq7.10}
(\frac{d}{ds})^m f(g(s))=\sum_{l=1}^m \sum_{b\in A_{m,l}} \frac{m!}{b_1! b_2!\cdots b_m!}f^{(l)}(g(s)(\frac{g'(s)}{1!})^{b_1}(\frac{g''(s)}{2!})^{b_2}
\cdots (\frac{g^{(m)}(s)}{m!})^{b_m}
\ee
where the  second summation is extended over all admissible partitions
\[A_{m,l}=\{b=(b_1,\ldots, b_m)\in \bfN^m: 
 b_1+b_2+\ldots +b_{m}=l,\,\, b_1+2b_2+\ldots+
mb_{m}=m\}~.\]

%The Bell polynomial version is as follows:
%\be \label{eq4.12}
%(\frac{d}{ds})^m f(g(s))=\sum_{l=1}^m  f^{(l)}(g(s)) B_{m,l}(g'(s),g''(s),\ldots,g^{(m-l+1)}(s))
%\ee
%where \be \label{eq4.13}
%B_{m,l}(x_1,\ldots,x_{m-l+1})=\frac1{l!}\sum_{j_1+\ldots+j_l=l, \,j_i\geq1}{m\choose j_1,\ldots,j_l}x_{j_1}\cdots
%x_{j_l},\,\,B_{0,0}=1~.
%\ee

An alternate  version (attributed by Johnson to J. F. C. Tiburce Abadie) is 
\be \label{eq7.14}
(\frac{d}{ds})^m f(g(s))=\sum_{l=1}^m {m \choose l} f^{(l)}(g(s))\left\{(\frac{d}{dh})^{m-l} (\frac{g(s+h)-g(s)}h)^l\right\}_{h=0}
\ee

The following corollary was rediscovered by Yamanaka \cite{Y}.
\begin{corollary} \label{corA1} \be \begin{aligned}
&(\frac{d}{ds})^m f(g(s))=\sum_{l=1}^m {m \choose l} f^{(l)}(g(s))\left\{(\frac{d}{dh})^{m-l}[\int_0^1g'(s+h\theta)~d\theta]^l\right\}_{h=0}\\
&=f'(g(s))g^{(m)}(s)+\sum_{l=2}^m {m \choose l} f^{(l)}(g(s))\left\{(\frac{d}{dh})^{m-l}[\int_0^1g'(s+h\theta)~d\theta]^l\right\}_{h=0}\\
\end{aligned}
\ee
\end{corollary}

We apply these formulas to our situation.
Let $\eta(x)=\frac1{1+4x^2}$, then
\be \label{eq20}
(\frac{d}{dx})^p \eta(x)=(-2)^p p! \Lambda_{p+1}(x),\, \,\Lambda_{p+1}(x):=\frac{  \sin{ [(p+1)\arcsin{\frac1{\sqrt{1+4x^2}}  }]}    }
{(1+4x^2)^{\frac{p+1}2}}
\ee
Hence also,
\be \label{eq30}
(\frac{d}{dx})^{p+1}\arctan{2x}=2(\frac{d}{dx})^p \eta(x)=(-1)^p 2^{p+1}p!\Lambda_{p+1}(x)
\ee

Therefore,
\be \label{eq35}
\begin{aligned}
&(\frac{d}{ds})^{m}\eta(s\gamma')=\sum_{l=1}^m\sum_{b\in A_{m,l}} \frac{m!}{b_1! b_2!\cdots b_m!}(-2)^l l! \Lambda_{l+1}(s\gamma') 
\cdot (\frac{s\gamma''+\gamma'}{1!})^{b_1}\cdots  (\frac{s\gamma^{(m+1)}+m\gamma^{(m)}}{m!})^{b_m}\\
&=-8s\gamma'\eta^2(s\gamma^{(m+1)}(s)+m\gamma^{(m)}(s))
+\sum_{l=2}^m\sum_{b\in A_{m,l}} \frac{m!}{b_1! b_2!\cdots b_{m-1}!}(-2)^l l! \Lambda_{l+1}(s\gamma') \\
&\cdot (\frac{s\gamma''+\gamma'}{1!})^{b_1}\cdots  (\frac{s\gamma^{(m)}+(m-1)\gamma^{(m-1)}}{(m-1)!})
^{b_{m-1}}~.
 \end{aligned}
\ee

\be \label{eq70}
\begin{aligned}
 (\frac{d}{ds})^m \arctan{(2s\gamma'(s))}&=\sum_{l=1}^m \sum _{b\in A_{m,l}}\frac{m!}{b_1! b_2!\cdots b_m!}(-1)^{l-1}2^l (l-1)! \Lambda_{l}(s\gamma')\\ 
&\cdot (\frac{s\gamma''+\gamma'}{1!})^{b_1}\cdots  (\frac{s\gamma^{(m+1)}+m\gamma^{(m)}}{m!})^{b_m}\\
 \end{aligned}
\ee

\be \label{eq40}
\begin{aligned}
\hspace{.1in} &(\frac{d}{ds})^k (s\frac{d}{ds})\arctan{(2s\gamma'(s)}(0)=(s\frac{d}{ds}+k)(\frac{d}{ds})^k \arctan{(2s\gamma'(s)}(0)\\
&=k(\frac{d}{ds})^k \arctan{(2s\gamma'(s)}(0)\\
&=k\sum_{l=1}^k  (-1)^{l-1}2^l (l-1)! \sum_{b\in A_{k,l}} \frac{k!}{b_1! b_2!\cdots b_k!}  
\cdot (\frac{\gamma'(0)}{1!})^{b_1}\cdots  (\frac{\gamma^{(k)}(0)}{(k-1)!})^{b_k}\\
 \end{aligned}
\ee

Similarly,
\be \label{eq2.15}
(\frac{d}{ds})^k H(\gamma)(0) =\sum_{l=1}^k H^{(l)}(0)\sum_{b \in A_{k,l}} \frac{k!}{b_1!\cdots b_k!}\\
\cdot (\frac{\gamma'(0)}{1!})^{b_1}(\frac{\gamma''(0)}{2!})^{b_2}\cdots (\frac{\gamma^{(k)}(0)}{k!})^{b_k}~.
\ee

%and so (recall \eqref{eq2.16}
%\be \label{eqA20}
%A^{\e}_{k+1}=\sum_{l=1}^k (H^{(l)}(0)-2k(-1)^{l-1}2^l (l-1)! )\sum_{b \in A_{k,l}} \frac{k!}{b_1!\cdots b_k!}\\
%\cdot (\frac{\gamma'(0)}{1!})^{b_1}(\frac{\gamma''(0)}{2!})^{b_2}\cdots (\frac{\gamma^{(k)}(0)}{k!})^{b_k}
%\ee

%\begin{lemma}\label{lemA1.5} Let $p\leq \min{(b,k-b)}$ where $b,\,p,\,k$ are positive integers. Then
%\[b! (k-b)!\leq p! (k-p)!~.\]
%\end{lemma}
%\begin{proof} The statement is equivalent to ${k\choose p} \leq {k \choose b}$ which is well-known to be correct for $p\leq \min{(b,k-b)}$.
%\end{proof}
%\begin{lemma} \label{lemA2} Let $f(x)$ be analytic and let $u(s)$ be smooth and satisfy \\
%\be
%\begin{aligned}
% &i.\, |(\frac{d}{ds})^{(j)}u(s)|\leq M^j j!^2 \,\,\text{on}\,\, [0,s_0],\, 1\leq j \leq m~,\, \text{or}\\
%&ii.\,\|(\frac{d}{ds})^{(j)}u(s)\|_{L^2[0,s_0]}\leq M^j j!^2\,\, \text{ on}\,\, [0,s_0],\, 1\leq j \leq m~.
%\end{aligned}
%\ee
%Then $|(\frac{d}{ds})^{(j)}f(u(s))|\leq CM^j j!^2$ on $[0,s_0]$, where $C$ depends only on $f$.\\
%\end{lemma}

\begin{lemma}\label{lemA2} 
\be \label{eqA40}
\begin{aligned}
(\frac{d}{ds})^n\eta(s\gamma')&=\sum_{l=1}^n {n\choose l}(-2)^l l! \Lambda_{l+1}(s\gamma'(s))\left\{ (\frac{d}{dh})^{n-l}[\int_0^1 \gamma''(s+h\theta)~d\theta +\gamma'(s+h)]^l\right\}_{h=0}\\
&=-8s\gamma'(s)\eta^2(s\gamma'(s))(\frac{d}{ds})^n[s\gamma']\\
&+ \sum_{l=2}^n {n\choose l}(-2)^l l! \Lambda_{l+1}(s\gamma'(s))\left\{ (\frac{d}{dh})^{n-l}[\int_0^1 \gamma''(s+h\theta)~d\theta +\gamma'(s+h)]^l\right\}_{h=0}.
\end{aligned}
\ee
\end{lemma}
\begin{proof} By Corollary \ref{corA1} and \eqref{eq20} it suffices to calculate $\int_0^1 g'(s+h\theta)~d\theta$ for $g(s)=s\gamma'(s),\, g'(s)=\gamma'(s)+s\gamma''(s)$, Hence
\[\int_0^1 g'(s+h\theta)~\theta=\int_0^1\{(s+h\theta)\gamma''(s+h\theta)+\gamma'(s+h\theta)\}~d\theta
=\int_0^1s\gamma''(s+h\theta)~d\theta+\gamma'(s+h)\]
after  an integration by parts.
\end{proof}

\begin{lemma}\label{lemA3}  Assume  $|(\frac{d}{ds})^j u(s)|\leq M^{j-1} \frac{j!^2}{(j+1)^2}$ on $[0,s_0]$ for $0\leq j \leq m$. Then 
\be \label{eqA50} |(\frac{d}{ds})^{j}(u(s))^{l} |\leq (\frac{C}{M})^{l-1} M^{j-1} \frac{j!^2}{(j+1)^2} \,\,
 \text{on} \,\,[0,s_0]~
 \ee
 for $0\leq j \leq m $ where $l$ is a positive integer and C is a universal constant.
\end{lemma}
\begin{proof} We use induction on $l$. The case $l=1$ is trivial.  Set $u_l(s)=u^l(s)$ and suppose that for \eqref{eqA50} holds for some positive integer $l$.
Then
\be \begin{aligned}
&|(\frac{d}{ds})^{j}u_{l+1}(s)|=|\sum_{i=0}^j {j\choose i}(\frac{d}{ds})^i u_l(s) (\frac{d}{ds})^{j-i}u(s)|\\
&\leq \sum_{i=0}^j {j\choose i}(\frac{C}{M})^{l-1} M^{i-1} \frac{i!^2}{(i+1)^2} M^{j-i-1}\frac{(j-i)!^2}{(j-i+1)^2}\\
&\leq  (\frac{C}{M})^{l-1} M^{j-2} m!\sum_{i=0}^j i! (j-i)! \frac1{(j+2)^2}(\frac1{i+1}+\frac1{j-i+1})^2\\
&\leq (\frac{C}{M})^{l} M^{j-1}\frac{j!^2}{(j+1)^2} \frac4{C}\sum_{i=0}^{\infty}\frac1{i^2}
\leq (\frac{C}{M})^{l} M^{j-1} \frac{j!^2}{(j+1)^2} ~,\\
&\text{for $C=4\sum_{i=0}^{\infty}\frac1{i^2}$}~.
\end{aligned} \ee
\end{proof}

\begin{lemma} \label{lemA4} Assume that 
\be \label{eqA60}
|(\frac{d}{ds})^{j+1}(s\gamma'(s))|\leq M^{j-1}\frac{j!^2}{j+1}
\ee
on $[0,s_0]$ for $j\geq 1$ and define
\[u(h;s)=\int_0^1 s\gamma''(s+h\theta)~d\theta+\gamma'(s+h),\, \, u_l(h;s)=(u(h;s))^l~.\]
Then  \be \label{eqA65}
|(\frac{d}{dh})^j u_l(h;s)|_{h=0}|\leq (\frac{C}{M})^{l-1}M^{j-1}\frac{j!^2}{(j+1)^2}~.
\ee
\end{lemma}
\begin{proof} Note that 
\be \label{eqA70}
\begin{aligned}
\hspace{.2in} &(\frac{d}{dh})^j u(h;s)=\int_0^1 s\theta^j \gamma^{(j+2)}(s+h\theta)~d\theta+\gamma^{(j+1)}(s+h)\\
&=\int_0^1 \theta^j (s+h\theta)\gamma^{(j+2)}(s+h\theta)~d\theta-\int_0^1 \theta^{j+1}\frac{d}{d\theta}[\gamma^{(j+1)}(s+h\theta)]~d\theta+\gamma^{(j+1)}(s+h)\\
&=\int_0^1\theta^j\{(s+h\theta)\gamma^{(j+2)}+(j+1)\gamma^{(j+1)}(s+h\theta)\}~d\theta\\
&=\int_0^1\theta^j (\frac{d}{ds})^{j+1}[s\gamma'(s)]( s+h\theta)~d\theta~.
\end{aligned}
\ee
Therefore \eqref{eqA70} and assumption \eqref{eqA60} imply
\[|(\frac{d}{dh})^j u(h;s)|\leq M^{j-1}\frac{j!^2}{j+1}\int_0^1 \theta^j~d\theta=M^{j-1}\frac{j!^2}{(j+1)^2}~.\]
Hence we can apply Lemma \ref{lemA3} to conclude \eqref{eqA65}.

\end{proof}

\end{document}